\documentclass[reqno]{amsart}
\usepackage{amsfonts}
\usepackage[active]{srcltx}
\usepackage{amsmath}
\usepackage{amssymb}
\usepackage{amsthm}
\usepackage{mathtools}
\usepackage{shuffle}
\usepackage{bm} %for bold math symbol

\usepackage{enumitem}
\usepackage{wasysym}
\usepackage{mathscinet}
\usepackage{verbatim}
\usepackage{graphicx}
\usepackage[
bookmarks=true,         %generates bookmarks for all entries in the "Table of Contents"
bookmarksnumbered=true, %includes chapter-/header-/section-/subsection-/... -
colorlinks=true, pdfstartview=FitV, linkcolor=blue, citecolor=blue,
urlcolor=blue]{hyperref}
\usepackage{microtype}
\theoremstyle{plain}
\newtheorem{theorem}{Theorem}[section]

\newtheorem{lemma}[theorem]{Lemma}
\newtheorem{problem}[theorem]{Problem}

\renewenvironment{proof}[1][Proof]{\textbf{#1.} }{\ \rule{0.5em}{0.5em} \par }
\theoremstyle{remark}

\theoremstyle{definition}

\newtheorem{definition}[theorem]{Definition}

\def\RR{\mathbb{R}}
\def\EE{\mathbb{E}}

\def\cF{{\mathcal F}}

\def\si{{\sigma}}

\def\Om{{\Omega}}

\def\al{{\alpha}}

\def\si{{\sigma}}

\def \eref#1{\hbox{(\ref{#1})}}

\def\EE{\mathbb{ E}}

\def\si{{\sigma}}
\def\al{{\alpha}}

\newtheorem{Rem}[theorem]{Remark}
\renewcommand{\SS}{\mathbb S}

\let\Section=\section
\def\section{\setcounter{equation}{0}\Section}

\title[Smoothness of density for SDEs with Markovian switching]{Smoothness of density for stochastic differential equations with Markovian switching}
\author[Y. Hu]{Yaozhong {\sc Hu}}\thanks{Y.  Hu is
partially supported by a grant from the Simons Foundation
\#209206.}
\address{Department of Mathematical  and Statistical Sciences\\
University of Alberta\\
 Edmonton, Alberta,Canada T6G 2G1}
\email{yaozhong@ualberta.ca}

\author[D. Nualart]{David {\sc Nualart}}\thanks{D. Nualart is
supported by the NSF grant DMS1512891. }
\address{Department of Mathematics \\
The University of Kansas \\
Lawrence, Kansas, 66045}
\email{nualart@math.ku.edu}

\author[X. Sun]{Xiaobin {\sc Sun}}\thanks{X. Sun is
supported by Natural Science Foundation of China (11601196), Natural Science Foundation of the Higher Education Institutions of Jiangsu Province (16KJB110006) and Scientific Research Staring Foundation of Jiangsu Normal University (15XLR010). }
\address{School of Mathematics and Statistics\\  Jiangsu Normal University\\  Xuzhou 221116, China}
\email{xbsun@jsnu.edu.cn}

\author[Y. Xie]{Yingchao {\sc Xie} }\thanks{Y. Xie is supported by Natural Science Foundation of China (11771187) and the
Project Funded by the Priority Academic Program Development of Jiangsu Higher Education Institutions. }
\address{School of Mathematics and Statistics\\  Jiangsu Normal University\\  Xuzhou 221116, China}
\email{ycxie@jsnu.edu.cn }

 \keywords{Malliavin calculus, Markovian switching, smoothness of density, Bismut formula,  strong Feller
property. }

\begin{document}

\begin{abstract}  This paper is concerned with a class of stochastic differential equations with Markovian switching.
The Malliavin calculus is used to study the smoothness of the density  of the solution under a H\"{o}rmander type condition. Furthermore, we obtain a
Bismut type formula which is used to establish the strong Feller property.
\end{abstract}

 \maketitle

\section{Introduction}

This paper considers  the  following stochastic differential equation with Markovian switching on $\RR^n$:
\begin{equation}
dX_t=b(X_t, \alpha_t)dt+\sigma(X_t, \alpha_t)dW_t, \ \ (X_0\,,
\al_0)=(x,\al)\in \RR^n \times \SS \,, \label{e.1.x}
\end{equation}
where $\mathbb{S}=\{1, 2, \ldots,
m_0\}$ and  $\left\{\alpha_t\,, t\ge 0\right\}$ is a right-continuous  $\SS$-valued Markov chain  described by
\begin{equation}
\mathbb{P}\{\alpha_{t+\Delta}=j|\alpha_{t}=i\}=\left\{\begin{array}{l}
\displaystyle q_{ij}\Delta+o(\Delta),~~~~~~i\neq j\\
1+q_{ii}\Delta+o(\Delta),~~~i=j,\end{array}\right.
\label{1.alpha}
\end{equation}
and $Q= (q_{ij})_{1 \le i,j \le m_0}$ is a $Q$-matrix.

Stochastic differential equations (SDEs) of this type have been extensively studied (see, for instance \cite{BBG,M,YZ,YM}). Existence and uniqueness of a solution,  existence of an invariant measure,  stability and other important properties have been analyzed.  In this paper we first study the smoothness of the density of the solution, then establish a Bismut  type  formula. Finally, as an application,  we prove the strong Feller property.

Our first purpose is to show the differentiability in the sense of  Malliavin calculus  of the solution  $X_t$  to equation  (\ref{e.1.x}). The difficulty here is the appearance of the Markovian
switching term $\alpha_t$, which is a jump process.  We will  perform perturbations of the underlying Brownian motion, keeping the Markovian switching process $\alpha_t$ unperturbed. The technique for this analysis is inspired in the partial Malliavin calculus, which can be regarded
as a stochastic calculus of variation for random variables with values in  a Hilbert space.

After developing  the Malliavin calculus for the solution $X_t$, we investigate the smoothness of the
density of the law of the  random vector $X_t$, for a fixed $t>0$,   with respect to the Lebesgue measure.
For this   we need to show that the determinant of the Malliavin
matrix of $X_t$ has negative moments of all orders. In the classical diffusion case, this is guaranteed by a  H\"{o}rmander  type nondegeneracy   condition.
To follow this classical approach we immediately encounter a difficulty: the process  $X_t$  depends on the discrete process
$\alpha_t$, and the application of It\^{o}'s formula yields some jump terms. To overcome this difficulty we shall use  the following strategy  inspired by \cite{FLT}.  First we notice that the jump times form a subset  of the jump times of some   Poisson process $N_t$, independent of the driving Brownian motion $W_t$. Then conditioning on $N_t=k$, there exists a random interval $[T_1, T_2]$   such that $T_2-T_1\geq\frac{t}{k+1}$. On this random time interval, the It\^{o}'s formula  for $(X_t, \alpha_t)$ will not produce a jump term, and  we can apply the  classical procedure. This requires a version of the classical Norris lemma on time intervals.

We also use Malliavin calculus to establish a Bismut  type  formula for the transition semigroup of the Markov process $(X_t, \alpha_t)$. As an application of Bismut  type  formula, we derive the strong Feller property of the process $(X_t, \alpha_t)$.

The paper is organized as follows. In the next section, we introduce  some notation
and  assumptions that we use throughout the paper. We develop  the Malliavin calculus for  SDEs with Markovian switching  in  Section 3.
In  Section 4, we show  that the determinant of the Malliavin covariance matrix
has all negative finite moments under a suitable uniform H\"{o}rmander's condition. Finally, we establish the Bismut  type  formula and use it to prove the strong Feller property in Section 5.

\section{Preliminaries}\label{sec.prelim}
Let  $(\Omega_1, \cF_1, \mathbb{P}_1)$ be the $d$-dimensional
canonical Wiener space with the natural filtration
$\cF_1= \left\{\cF_1(t),   t\ge 0\right \} $.
That is,  $\Omega_1$  is  the set of all continuous maps $\omega_1: \mathbb{R}_+\rightarrow \mathbb{R}^d$ such that $\omega_1(0)=0$  and    $\cF_1$
is the completion of the Borel $\sigma$ field of $\Omega_1$ with
respect to $\mathbb{P}_1$, where  $\mathbb{P}_1$ is the canonical
Wiener measure. Then, $W=\{ W_t(\omega_1):=\omega_1(t), t \ge 0\}$ is a
$d$-dimensional Brownian motion.

Let $\mathbb{S}=\{1, 2, \ldots,
m_0\}$,  where $m_0$ is a given positive integer which will be fixed
throughout the paper.
Let
$Q=\left(q_{ij}\right)_{1\le i,j\le m_0}$ be a $Q$-matrix satisfying the following assumption:
\begin{itemize}
\item[(i)] $q_{ij}\geq0$ for $i\neq j$,
\item[(ii)] $q_{ii}=-\sum_{j\neq i}q_{ij}$ for $i\in \mathbb{S}$,
\item[(iii)] $\sup_{i, j\in\mathbb{S}}|q_{ij}|:= K<\infty$.
\end{itemize}

Let
$(\Omega_2,\cF_2,\mathbb{P}_2)$ be another complete probability
space with a filtration $\cF_2=\{\cF_2(t) ,t\ge 0\}$ satisfying the usual conditions,
on which there exists a right-continuous $\SS$-valued Markov chain $\alpha=\{\alpha_t, t\ge 0\}$
satisfying  (\ref{1.alpha}).

Denote  the product probability space  by $(\Omega,
\cF, \mathbb{P}):=(\Omega_1\times\Omega_2, \cF_1\times\cF_2,
\mathbb{P}_1\times\mathbb{P}_2)$ with the product filtration $\cF =\{ \cF_t, t\ge 0\}$, where
$\cF_t=\cF_1(t)\times \cF_2(t)$.
We extend  $W_t$ and $\alpha_t$ to random variables defined on
 $\Omega$ by letting
$W_t(\omega)=\omega_{1}(t)$ and $\alpha_t(\omega)=\alpha_t(
\omega_2)$, respectively, if $\omega=(\omega_1,\omega_2)$.  Notice that on the probability space $(\Om,\cF, \mathbb{P})$, the processes $W$ and
$\alpha$ are independent.

\vspace{0.3cm}
It is well known (see \cite{BBG}) that the process
$\alpha$ can   be  described in the
following manner. Introduce the function $g: \mathbb{S}\times [0,
m_0(m_0-1)K]\rightarrow\mathbb{R}$ defined by
\[
g(i, z)=\sum_{j\in \SS\backslash i }(j-i)1_{z\in\triangle_{ij}},\quad   i\in \SS\,,
\]
where $\triangle_{ij}$ are the consecutive (with respect to the
lexicographic ordering on $\mathbb{S}\times\mathbb{S}$) left-closed,
right-open intervals of $\mathbb{R}_{+}$, each having length
$q_{ij}$, with $\Delta_{12}= [0,q_{12})$. Then, equation (\ref{1.alpha}) can also be written as
\begin{equation} \label{2.3}
d\alpha(t)=\int_{[0, m_0(m_0-1)K]}g(\alpha_{t-}, z)N(dt, dz),
\end{equation}
where $N(dt, dz)$ is a Poisson random measure defined on $\Omega\times\mathcal{B}(\mathbb{\RR_{+}})\times\mathcal{B}(\mathbb{\RR_{+}})$, whose intensity measure is Lebesgue measure, and $N(dt, dz)$ is independent of $W$.

\vspace{0.3cm}
For $k\in\mathbb{N}$ we denote by $C^{k}(\mathbb{R}^{n}\times\mathbb{S};\mathbb{R}^n)$
the family of all $\mathbb{R}^n$-valued functions $f(x, \alpha)$ on $\mathbb{R}^{n}\times\mathbb{S}$ which
are $k$-times continuously differentiable in $x$ for any $\alpha \in \mathbb{S}$.  The  $k$-th derivative tensor of $f$ with respect to $x$ is denoted by $\nabla^k f(x, \alpha)$.

We denote by $|\cdot |$ the Euclidean norm and consider the metric  $\Lambda$ on $\mathbb{R}^n\times\mathbb{S}$\  given by
 $\Lambda((x, i), (y, j))=|x-y|+d(i, j)$,  for $x, y\in\mathbb{R}^n, i, j\in\mathbb{S}$,
where $d(i, j)=0$ if $i=j$ and $d(i, j)=1$ if $i\neq j$. Let ${\mathcal B}_b(\mathbb{R}^n\times\mathbb{S})$
be the family of all bounded Borel measurable functions on $\mathbb{R}^n\times\mathbb{S}$.

Suppose that $b  :\RR^n\times \SS \rightarrow \RR^n$ and $ \sigma
:\RR^n\times \SS \rightarrow \RR^{nd }$  are functions satisfying the  following assumptions:

\smallskip
\noindent
\textbf{(${\bf A}_1$)} There is a positive constant $C_1$ such that
\[
|b(x, i)-b(y, i)|\vee|\sigma(x, i)-\sigma(y, i)|\leq C_1|x-y|
\quad\hbox{for any $x,y\in\mathbb{R}^{n}, i\in\mathbb{S}$}\,.
\]

\smallskip
\noindent
\textbf{(${\bf H}_k$)} Fix an integer $k\ge 1$. For each $i=1, \dots, d$, the functions $b$ and $\sigma_i$ belong to
$C^{k}(\mathbb{R}^{n}\times\mathbb{S};\mathbb{R}^n)$, and have bounded partial derivatives up to the order $k$.

\medskip
\noindent \textbf{(${\bf H}_\infty$)}   For each $i=1, \dots, d$ and for any $k\ge 1$, the functions $b$ and $\sigma_i$ belong to
$C^{k}(\mathbb{R}^{n}\times\mathbb{S};\mathbb{R}^n)$, and have bounded partial derivatives of all orders.

\medskip
It is clear that  \textbf{(${\bf H}_k$)} implies \textbf{(${\bf A}_1$)} for any $k\ge 1$.

\medskip
Under  condition \textbf{(${\bf A}_1$)}, equation (\ref{e.1.x}) possesses a unique strong solution $X=\{ X_t, t\ge 0\}$.
Moreover, for any $p\ge 2$ and $T>0$, $\mathbb{E}(\sup_{0\leq t\leq T}|X_t|^p)\leq C_2$, where $C_2$ is a positive constant depending only on
$p, T$ and $x$.  On the  other hand, $(X,\alpha)=\{(X_t,\alpha_t), t\ge 0\}$  is an homogeneous Markov process and the associated Markov semigroup $P_t$ satisfies
$$
P_tf(x, \alpha)=\mathbb{E}f(X_t(x,\alpha), \alpha_t(x, \alpha)), \quad  t\ge 0, f\in {\mathcal B}_b(\mathbb{R}^n\times\mathbb{S}).
$$
We refer, for instance,  to \cite{M} and \cite{YM} for  a detailed presentation and proofs of the above results.

Along the paper $C$ will denote a generic constant which may vary from line to line and it might depend on $T$, the exponent $p\ge 2$, the initial condition $x$ and a fixed element $h\in H$ (the precise definition of $H$ is in the next section).

\section{The Malliavin calculus}

In this section we analyze the regularity, in the sense of Malliavin calculus, of the     solution $X_t$
to  equation (\ref{e.1.x}). The procedure is to perform perturbations of the underlying Brownian motion, keeping the Markovian switching process $\alpha_t$ invariant. The technique for this analysis is inspired in the partial Malliavin calculus which can be regarded as a stochastic calculus of variation for random variables with values on  the Hilbert space $L^2(\Omega_2)$.

We follow the notation introduced in the Preliminaries.  Denote by $H$ the Hilbert
space $H=L^{2}(\mathbb{R}_+; \mathbb{R}^{d})$, equipped with the inner
product $\langle h_1, h_2\rangle_H=\int^{\infty}_{0} \langle
h_1(s),h_2(s)\rangle_{\RR^d} ds$.

For a  Hilbert space $U$ and a real number $p\ge 1$, we denote by
$L^p(\Omega_1; U)$ the space of $U$-valued random variables $\xi$
such that $\mathbb{E}_1\|\xi\|^{p}_{U}<\infty$, where $\mathbb{E}_1$
is the mathematical expectation on the probability space
$(\Omega_1, \cF_1, \mathbb{P}_1)$. We also set
$L^{\infty-}(\Omega_1; U):=\cap_{p<\infty}L^p(\Omega_1; U)$.

We introduce the derivative operator for a random variable $F$ in the space
$L^{\infty-}(\Omega_1; U)$ following the approach of Malliavin in  \cite{Ma}.
We say that $F$ belongs to $\mathbb{D}^{1,\infty}(U)$ if
there exists $DF\in L^{\infty-}(\Omega_{1}; H\otimes U)$ such that for any $h\in H$,
$$
\lim_{\varepsilon \rightarrow 0}\mathbb{E}_1\left\|\frac{F(\omega_1+\varepsilon \int_0^{\cdot} h_s ds)
-F(\omega_1)}{\varepsilon }-\langle DF, h\rangle_H\right\|^{p}_{U}=0
$$
holds for every $p\geq 1$. In this case, we define the Malliavin derivative of $F$
in the direction $h$ by $D^{h}F :=\langle DF, h\rangle_{H}.$
Then, for any $p\ge 1$ we define  the Sobolev space $\mathbb{D}^{1,p}(U)$ as the
completion of $\mathbb{D}^{1,\infty}(U)$ under the following norm
$$
\|F\|_{1, p, U}=\left[\mathbb{E}_1\|F\|^p_{U}\right]^{1/p}
+\left[\mathbb{E}_1\|DF\|^p_{H\otimes U}\right]^{1/p}.
$$
By induction we define the $k$-th derivative by $D^{k}F=D(D^{k-1}F),$
which is a random element with values in  $H^{\otimes k}\otimes U$.
For any integer  $k\ge 1$, the Sobolev space  $\mathbb{D}^{k,p}(U)$ is the
completion of  $\mathbb{D}^{k,\infty}(U)$ under the norm
$$
\|F\|_{k, p, U}=\|F\|_{k-1, p, U}+\|D^{k}F\|_{1, p, H^{\otimes k}\otimes U}.
$$
It turns out that  $D$ is a closed operator from $L^{p}(\Omega_{1};U)$ to $L^{p}(\Omega_{1}; H\otimes U)$.
Its adjoint $\delta$ is called the divergence operator, and is continuous form $\mathbb{D}^{1,p}( H\otimes U)$
to $L^{p}(\Omega_{1};U)$  for any $p>1$. The duality relationship reads
\[
\mathbb{E}_1(\langle DF, u\rangle_{H\otimes U})=\mathbb{E}_1(\langle F,\delta(u)\rangle_{U}),
\]
for any $F\in  \mathbb{D}^{1,p}( U)$ and $u\in \mathbb{D}^{1,q}(  H\otimes U)$, with $\frac 1p +\frac 1q=1$.

A square integrable
random variable $F\in L^2(\Omega)$ can be identified with an element
of $L^2(\Omega_1; V)$, where $V=L^2(\Omega_2)$.

The following is the main result of this section.
\begin{theorem}\label{main}   Suppose that  Hypothesis \textbf{(${\bf H}_2$)} holds.
Then for any $t\ge 0$ and any $h\in H$,
$X_t\in \mathbb{D}^{1,\infty}(\mathbb{R}^n\otimes V)$ and $D^hX_t$ satisfies
\begin{equation}\left\{\begin{array}{l}
\displaystyle dD^hX_t=\nabla b(X_t, \alpha_t)D^hX_tdt+\sum^{d}_{i=1}\nabla\sigma_{i}(X_t, \alpha_t)D^hX_tdW^{i}_t+\sigma(X_t, \alpha_t)h_tdt,\\
D^hX_0=0.\end{array}\right. \label{3.1}
\end{equation}
\end{theorem}
To prove the theorem,  let $X^{\varepsilon  h}_t$ be the solution of equation (\ref{e.1.x}) with $W_t$ replaced by $W_t+\varepsilon \int^t_0 h_s ds$,  where $ \varepsilon \in (0,1)$, that is,
\begin{equation}\left\{\begin{array}{l}
\displaystyle dX^{\varepsilon  h}_t=b(X^{\varepsilon  h}_t, \alpha_t)dt+
 \sigma (X^{\varepsilon  h}_t, \alpha_t)dW _t+\varepsilon \sigma(X^{\varepsilon  h}_t, \alpha_t)h_tdt,\\
(X^{\varepsilon  h}_0, \alpha_0)=(x, \alpha)\in\mathbb{R}^n\times\mathbb{S},\end{array}\right.
\label{e.3.5a}
\end{equation}
where $\alpha_t$ is defined in (\ref{1.alpha}). Then, we can write
\begin{eqnarray*}
\frac{X^{\varepsilon  h}_{t}-X_t}{\varepsilon }
=\!\!\!\!\!\!\!\!&&\frac{1}{\varepsilon }\int^t_0[b(X^{\varepsilon
h}_s, \alpha_s)-b(X_s, \alpha_s)]ds
+\frac{1}{\varepsilon }\int^t_0[\sigma(X^{\varepsilon  h}_s, \alpha_s)-\sigma(X_s, \alpha_s)]dW_s\nonumber\\
&&+\int^t_0\sigma(X^{\varepsilon  h}_s, \alpha_s)h_sds\nonumber.
\end{eqnarray*}

In order to prove Theorem \ref{main}, we first give some preliminary lemmas.
\begin{lemma} \label{l.3.2} Suppose that   Hypothesis \textbf{(${\bf A}_1$)}  holds.
Then for any $h\in H$, $T>0$ and $p\geq 2$, we have
$$
\mathbb{E}\left[\sup_{0\leq t\leq T}|X^{\varepsilon  h}_t|^{p}\right]\leq C.
$$
\end{lemma}
\begin{proof}
From equation \eref{e.3.5a}   it is easy to see that
\begin{eqnarray*}
|X^{\varepsilon  h}_t|^{p}
\le\!\!\!\!\!\!\!\!&& C\Bigg[|x|^{p}+\left|\int^t_0 b(X^{\varepsilon  h}_s, \alpha_s)ds\right|^{p}
+\left|\int^t_0\sigma(X^{\varepsilon  h}_s, \alpha_s)dW_s\right|^{p}\\
&&+\varepsilon^p\left|\int^t_0\sigma(X^{\varepsilon  h}_s, \alpha_s)h_sds\right|^{p}\Bigg]:=C \Bigg[ |x|^p+I_1(t)+I_2(t)+I_3(t)\Bigg] \,.
\end{eqnarray*}
By H\"{o}lder's  and  Burkholder-Davis-Gundy's inequalities, we obtain
\[
\mathbb{E}\left[\sup_{0\leq t\leq T} \left( I_1(t) + I_2(t) + I_3(t)\right)  \right]
\leq C \int^T_0(\mathbb{E}|X^{\varepsilon  h}_s|^{p}+1)ds \,.
\]
Then the desired estimate follows from Gronwall's lemma.
\end{proof}

\begin{lemma} \label{l.3.5}Suppose that  Hypothesis \textbf{(${\bf A}_1$)} holds.
Then for any $h\in H$, $T>0$  and $p\geq 2$, we have
\[
\mathbb{E}\left[\sup_{0\leq t\leq T}|X^{\varepsilon  h}_t-X_t|^{p}\right]\leq C\varepsilon ^{p}.
\]
\end{lemma}
\begin{proof}We write
\begin{eqnarray}
X^{\varepsilon  h}_t-X_t
=\!\!\!\!\!\!\!\!&&\int^t_0 [b(X^{\varepsilon  h}_s, \alpha_s)-b(X_s, \alpha_s)]ds
+\varepsilon \int^t_0\sigma(X^{\varepsilon  h}_s, \alpha_s)h_sds \nonumber\\
&&+\int^t_0[\sigma(X^{\varepsilon  h}_s, \alpha_s)-\sigma(X_s, \alpha_s)]dW_s. \nonumber
\end{eqnarray}
Applying  H\"{o}lder's inequality, Burkholder-Davis-Gundy's inequality and Lemma \ref{l.3.2},  we have
\begin{eqnarray*}
\mathbb{E}\left[\sup_{0\leq t\leq T}|X^{\varepsilon  h}_t-X_t|^{p}\right]
\leq\!\!\!\!\!\!\!\!&&C \int^T_0\mathbb{E}|X^{\varepsilon  h}_s-X_s|^{p}ds+C \varepsilon ^{p} .
\end{eqnarray*}
Hence,  Gronwall's inequality yields
$$
\mathbb{E}\left[\sup_{0\leq t\leq T}|X^{\varepsilon  h}_t-X_t|^{p}\right]\leq C \varepsilon ^p.
$$
\end{proof}

\medskip
\noindent
\textit{Proof of Theorem \ref{main}} \quad
Let $\psi^{h}_t$ be the solution of equation \eref{3.1}. It is easy to show  that $\mathbb{E}\left[\sup_{0\leq t\leq T}|\psi^{h}_t|^{p}\right]\leq C$,
where $C$ is a constant   depending on $T, x, h$ and $ p$.  We have
\begin{eqnarray*}
 \frac{X^{\varepsilon  h}_{t}-X_t}{\varepsilon }-\psi^{h}_t
=\!\!\!\!\!\!\!\!&&\frac{1}{\varepsilon } \int^t_0 [b(X^{\varepsilon  h}_s, \alpha_s)
-b(X_s, \alpha_s)\!\!-\!\!\varepsilon \nabla b(X_s, \alpha_s)\psi^{h}_s]ds
 \nonumber\\
&&+\frac{1}{\varepsilon }\!\!\int^t_0\!\!\sum^{d}_{i=1}[\sigma_{i}(X^{\varepsilon  h}_s, \alpha_s)\!\!-\sigma_{i}(X_s, \alpha_s)\!\!
-\!\!\varepsilon \nabla\sigma_{i}(X_s, \alpha_s)\psi^{h}_s]dW^{i}_s\nonumber\\
&&+\int^t_0\!\![\sigma(X^{\varepsilon  h}_s, \alpha_s)-\sigma(X_s, \alpha_s)]h_sds\,.
\end{eqnarray*}
Using twice the mean valued theorem we have
\begin{eqnarray*}
&&\frac{X^{\varepsilon  h}_{t}-X_t}{\varepsilon }-\psi^{h}_t\\
=\!\!\!\!\!\!\!\!&&\int^t_0  \left[  \left( \int^1_0\nabla b(X_s+\nu(X^{\varepsilon  h}_s-X_s), \alpha_s)d\nu \right)
\frac{X^{\varepsilon  h}_s-X_s}{\varepsilon }-\nabla b(X_s, \alpha_s)\psi^{h}_s \right]ds\nonumber\\
&&+\!\int^t_0\!\sum^{d}_{i=1}\!\left[ \left( \!\int^1_0\!\nabla\sigma_{i}(X_s\!+\!\nu(X^{\varepsilon  h}_s\!-\!X_s), \alpha_s)d\nu \right)
\frac{X^{\varepsilon  h}_s-X_s}{\varepsilon }-\nabla\sigma_{i}(X_s, \alpha_s)\psi^{h}_s\right] dW^{i}_s\nonumber\\
&&+\int^t_0\left[\sigma(X^{\varepsilon  h}_s, \alpha_s)-\sigma(X_s, \alpha_s)\right]h_sds\nonumber\\
=\!\!\!\!\!\!\!\!&&\int^t_0 \left( \int^1_0\nabla b(X_s+\nu(X^{\varepsilon  h}_s-X_s), \alpha_s)d\nu \right)
\left(\frac{X^{\varepsilon  h}_s-X_s}{\varepsilon }-\psi^{h}_s\right)ds\nonumber\\
&&+\int^t_0\sum^{d}_{i=1}\left( \int^1_0\nabla\sigma_{i}(X_s+\nu(X^{\varepsilon  h}_s-X_s), \alpha_s)d\nu \right)\left(\frac{X^{\varepsilon  h}_s-X_s}{\varepsilon }-\psi^{h}_s\right)
 dW^{i}_s+\varphi^{\varepsilon h}_t,\nonumber
\end{eqnarray*}
where $\varphi^{\varepsilon h}_t$ defined by
\begin{eqnarray*}
\varphi^{\varepsilon h}_t=\!\!\!\!\!\!\!\!&&\int^t_0[\sigma(X^{\varepsilon  h}_s, \alpha_s)-\sigma(X_s, \alpha_s)]h_sds\\
&&
+\int^t_0\left(\int^1_0\nabla b(X_s+\nu(X^{\varepsilon  h}_s-X_s), \alpha_s)d\nu-\nabla b(X_s, \alpha_s)\right)\psi^{h}_sds\\
&&+\int^t_0\sum^{d}_{i=1}\left(\int^1_0\nabla\sigma_{i}(X_s+\nu(X^{\varepsilon  h}_s-X_s), \alpha_s)d\nu-\nabla\sigma_{i}(X_s, \alpha_s)\right)\psi^{h}_s dW^{i}_s.
\end{eqnarray*}
By   Hypothesis  \textbf{(${\bf H}_2$)}, we obtain
\begin{eqnarray*}
\mathbb{E}\left[\sup_{0\leq s\leq t}\left|\frac{X^{\varepsilon  h}_{s}-X_s}{\varepsilon }-\psi^{h}_s\right|^p\right]
\le\!\!\!\!\!\!\!\!&&C \int^t_0\mathbb{E}\left|\frac{X^{\varepsilon  h}_s-X_s}{\varepsilon }-\psi^{h}_s\right|^p ds\\
&& +C \mathbb{E}\left[\sup_{0\leq s\leq t}|X^{\varepsilon  h}_s-X_s|^{p}\right] \\
&&+C\left(\mathbb{E}\sup_{0\leq s\leq t}|X^{\varepsilon  h}_s-X_s|^{2p}\right)^{1/2}
\left(\mathbb{E}\sup_{0\leq s\leq t}|\psi^{h}_s|^{2p}\right)^{1/2}.
\end{eqnarray*}
Using   Gronwall's inequality and   Lemma \ref{l.3.5}, we
obtain
$$\lim_{\varepsilon \rightarrow 0}\mathbb{E}\left[\sup_{0\leq s\leq t}
\left|\frac{X^{\varepsilon  h}_{s}-X_s}{\varepsilon }-\psi^{h}_s\right|^p\right]=0.$$
This implies that for  any $p\geq 2$,
$$
\lim_{\varepsilon \rightarrow 0}\mathbb{E}_{1}\left[\sup_{0\leq s\leq t}
\left\|\frac{X^{\varepsilon  h}_{s}-X_s}{\varepsilon }-\psi^{h}_s\right\|_{\mathbb{R}^n\otimes V}^p\right]=0.
$$
Now, let $D_{s}X_t$ be the solution of the following equation:
\begin{eqnarray*}
D_{s}X_t=\!\!\!\!\!\!\!\!&&\sigma(X_s, \alpha_s)+\int^t_{s}\nabla b(X_r, \alpha_r)D_{s}X_r dr
+\int^t_{s}\sum^{d}_{i=1}\nabla\sigma_{i}(X_r, \alpha_r)D_{s}X_r dW^{i}_r
\end{eqnarray*}
for $s\leq t$ and $D_{s}X_t=0$ for $s>t$.
Then we can easily obtain that $D^hX_t=\psi^{h}_t$ and
$DX_t\in L^{\infty-}(\Omega_{1}, H\otimes\mathbb{R}^n\otimes V)$.
The proof is complete.  \quad $\blacksquare$

As a consequence, we can show the following version of the chain rule.

\begin{theorem} \textbf{(Chain rule)} Assume that  condition \textbf{(${\bf H}_2$)} holds.
Then for any $h\in H$, $t\ge 0$ and  $p\geq 2$, if  $f\in C^{2}_b(\mathbb{R}^n\times\mathbb{S})$, we have
$$
\lim_{\varepsilon \rightarrow 0}\mathbb{E}\left|\frac{f(X^{\varepsilon  h}_{t}, \alpha_t)
-f(X_t, \alpha_t)}{\varepsilon }-\nabla f(X_t, \alpha_t)D^{h}X_t\right|^p=0.
$$
Moreover, $f(X_t, \alpha_t)\in \mathbb{D}^{1, \infty}(V)$ and $Df(X_t, \alpha_t)=\nabla f(X_t, \alpha_t)DX_t$.
\end{theorem}
\begin{proof}
We can write
\begin{eqnarray}
&&\mathbb{E}\left|\frac{f(X^{\varepsilon  h}_{t}, \alpha_t)-f(X_t, \alpha_t)}{\varepsilon }-\nabla f(X_t, \alpha_t)D^{h}X_t\right|^p\nonumber\\
&& \qquad \leq
 C\mathbb{E}\left|\frac{f(X^{\varepsilon  h}_{t}, \alpha_t)
-f(X_t, \alpha_t)}{\varepsilon }-\nabla f(X_t, \alpha_t)\frac{X^{\varepsilon  h}_{t}-X_t}{\varepsilon }\right|^p\nonumber\\
&&\qquad\qquad +C\mathbb{E}\left|\nabla f(X_t, \alpha_t)\frac{X^{\varepsilon  h}_{t}-X_t}{\varepsilon }
-\nabla f(X_t, \alpha_t)D^{h}X_t\right|^p. \label{e.3.15}
\end{eqnarray}
For the second term in \eref{e.3.15}, by Theorem  \ref{main} we have
\begin{equation}
\lim_{\varepsilon \rightarrow 0}\mathbb{E}\left|\nabla f(X_t, \alpha_t)\frac{X^{\varepsilon  h}_{t}-X_t}{\varepsilon }
-\nabla f(X_t, \alpha_t)D^{h}X_t\right|^p=0.
\end{equation}
Consider now the first term in \eref{e.3.15}, since $f\in C^{2}_b(\mathbb{R}^n\times\mathbb{S})$ we have
\begin{eqnarray}
&&\mathbb{E}\left|\frac{f(X^{\varepsilon  h}_{t},\alpha_t)
-f(X_t,\alpha_t)}{\varepsilon }-\nabla f(X_t, \alpha_t)\frac{X^{\varepsilon  h}_{t}-X_t}{\varepsilon }\right|^p\nonumber\\
&&\qquad\qquad \leq  C\mathbb{E}\frac{|X^{\varepsilon  h}_{t}-X_t|^{2p}}{\varepsilon ^p}\leq C\varepsilon ^p. \label{e.3.19}
\end{eqnarray}
From  \eref{e.3.15}-\eref{e.3.19} the theorem follows.
\end{proof}

 \section{Smoothness of the density}
 \setcounter{equation}{0}

In this section we show that under suitable non degeneracy assumptions on the coefficients, for any $t>0$ the    random vector $X_t$ has a smooth density.
 To this end  we first study  the stochastic flow associated with equation \eref{e.1.x} and then we show
that the determinant of the Malliavin covariance matrix of $X_t$ has    finite negative moments of all orders  for any $t>0$, under
a uniform H\"{o}rmander's condition.

\begin{definition}
Suppose that $F(x, \alpha):\Om\rightarrow \RR^n$ is a measurable  function
for all $x\in \RR^n$ and $\al\in \SS$.  We say that its  gradient
  with respect to $x$ exists  (in mean square sense)
if there is $A(x,\al): \Om \rightarrow \RR^{n^2}$ such that
for any $\xi\in \RR^n$ we have
\[
\lim_{\varepsilon \rightarrow 0}\mathbb{E}\left|\frac{F(x+\varepsilon \xi, \alpha)
-F(x, \alpha)}{\varepsilon }-A(x,\al)\xi \right|^{2}=0\,.
\]
We denote the gradient matrix $A(x,\al)$ by $\nabla  F(x, \alpha)$.
\end{definition}

By arguments similar to those used in the proof of   Theorem \ref{main},  we can
 obtain the following results.

\begin{theorem}
Assume the hypothesis \textbf{(${\bf H}_2$)} holds. Let $\left\{(X_{s, t}(x, \alpha),
\alpha_{s, t}(\alpha))\,, t\ge s\right\}$ be the solution of equations (\ref{e.1.x}) and (\ref{1.alpha}),
which starts from
$(x,\al)$ at time $s$. Then  the gradient  of $X_{s, t}(x, \alpha)$
with respect to $x$  (in mean square) exists.  If  we denote
\[
J_{s,t}  := \nabla X_{s, t}(x, \alpha)\,,
\]
then
\begin{equation}
\begin{cases}
\displaystyle dJ_{s,t}  =\nabla b(X_t, \alpha_t)J_{s,t}  dt
+\sum^{d}_{i=1}\nabla\sigma_{i}(X_t, \alpha_t)J_{s,t} dW^{i}_t,\quad t\ge s \\
J_{s,s}  =I\,,\label{e.4.1}
\end{cases}
\end{equation}
where $I$ is the $n$-dimensional identity matrix. Moreover, $J_{s,t}$ is invertible and its inverse $J_{s,t}^{-1}$
satisfies
\begin{equation}
\begin{cases}
\displaystyle dJ^{-1}_{s,t} = -J^{-1}_{s,t}\left(\nabla b(X_t, \alpha_t)-\sum^d_{i=1}\nabla\sigma_{i}(X_t, \alpha_t)
\nabla\sigma_{i}(X_t, \alpha_t)\right)dt   \\
\displaystyle \ \qquad \quad  -\sum^d_{i=1}J^{-1}_{s,t}\nabla\sigma_{i}(X_t, \alpha_t)dW^i_t,\\
J^{-1}_{s,s}= I.
\end{cases}
\end{equation}
\end{theorem}

 The following lemma provides estimates on the $L^p$-norm of the gradient of the solution and its inverse.

\begin{lemma} \label{l.4.3}Assume that Hypothesis \textbf{(${\bf H}_2$)}   holds. Then for any $p\geq 2$,
there exists a positive constant  $  C $  depending only on $T$ and $p$
such that
\[
\mathbb{E}\left[\sup_{s\leq t\leq T}\left(|J_{s,t}|^{p}+|J^{-1}_{s,t}|^{p}\right)
\right]\leq C \,.
\]
\end{lemma}
\begin{proof}
From
\[
J_{s,t} =I+\int^t_s \nabla b(X_r, \alpha_r)J_{s,r}  dr+\int^t_s \nabla\sigma(X_r, \alpha_r)J_{s,r}  dW_r.\,
\]
we obtain
$$
|J_{s,t} |^{p}\le C\left[1+ \left |\int^t_s \nabla b(X_r, \alpha_r)J_{s,r}  dr\right|^p
+\left|\int^t_s \nabla\sigma(X_r, \alpha_r)J_{s,r}  dW_r\right|^p\right]
$$
for any $p\geq2$. Then by H\"{o}lder's inequality, we can write
$$
\mathbb{E}\left[\sup_{s\leq t\leq T}\left|\int^t_s \nabla b(X_r, \alpha_r)J_{s,r}  dr\right|^p\right]
\leq C\int^T_s \mathbb{E}|J_{s,r} |^p dr,
$$
and by  Burkholder-Davis-Gundy's inequality we have
$$\mathbb{E}\left[\sup_{s\leq t\leq T}\left|\int^t_s\nabla\sigma(X_r, \alpha_r)J_{s,r}  dW_r\right|^{p}\right]
\leq C\int^T_s\mathbb{E}|J_{s,r} |^{p}ds.$$
Hence, we have
$$\mathbb{E}\left[\sup_{s\leq t\leq T}|J_{s,t} |^{p}\right]
\leq C +C\int^T_s\mathbb{E}|J_{s,r} |^{p}ds.
$$
Gronwall's inequality yields
$$\mathbb{E}\left[\sup_{s\leq t\leq T}|J_{s,t} |^{p}\right]\leq C  .$$
Similarly,
$$\mathbb{E}\left[\sup_{s\leq t\leq T}|J^{-1}_{s,t} |^{p}\right]\leq C  .$$
The proof  of the lemma is now complete.
\end{proof}

\vspace{0.3cm}

Using  the gradient of the flow  $J_{s,t} $ we can represent the Malliavin derivative $D X_t$ as follows:
\[
D_{s}X_t=J_{s,t}\sigma(X_s, \alpha_s), \quad 0\le s\leq t\le T\,;\quad \quad
 D_{s}X_t=0, s>t \,.
  \]

Next, we shall  study the Malliavin differentiability  of $J_{s, t}$. Denote by  by $D^i_r$  the Malliavin derivative with respect to the $i$-th component of the Brownian motion $W$  at time $r$.

\begin{lemma} \label{l.4.4} Suppose that  Hypothesis \textbf{(${\bf H}_3$)} holds. Then the following two statements hold:
\begin{enumerate}[label=(\roman*)]
\item   For all $0\leq s\leq t\leq T$,  $J_{s,t}\in \mathbb{D}^{1, \infty}(\mathbb{R}^n\otimes\mathbb{R}^n\otimes V)$  and
 for any $p\geq 2$,
there exists a positive constant  $C $  depending on $T$, $p$ and $x$, such that for all $i=1,\dots, d$ and  $r\in [0,T]$
\[
\mathbb{E}\left[\sup_{s\leq t\leq T}|D^i_r J_{s,t}|^{p}\right]\leq C.
\]
\item For any  $t\leq T$,  $X_t\in \mathbb{D}^{2, \infty}(\mathbb{R}^n\otimes V)$ and
  for any $p\geq 2$, there exists a positive constant $C$ depending on $T$, $p$ and $x$,
such that for all $i, j=1,\dots, d$ and  $r,s\le t $
\[
\mathbb{E}|D^i_r(D^{j}_{s}X_t)|^{p}\leq C.
\]
\end{enumerate}
\end{lemma}
\begin{proof} First, we prove (i).
By equation \eref{e.4.1},  the definition of Malliavin derivative  and the arguments  used in  the proof of  Theorem 3.1,
we obtain
$$
\lim_{\varepsilon \rightarrow 0}\mathbb{E}_1
\left\|\frac{J_{s,t}(\omega_1+\varepsilon \int^{\cdot}_0 h_s ds, \omega_2)
-J_{s,t}(\omega_1, \omega_2)}{\varepsilon }-\langle DJ_{s,t}, h
\rangle_H\right\|^{p}_{\mathbb{R}^n\otimes\mathbb{R}^n\otimes V}=0
$$
for any $h\in H$ and any $p\geq 2$, where $DJ_{s,t}$ satisfies the following equation for all $r\leq t$
\begin{eqnarray}
D^{i}_{r}J_{s,t}=\!\!\!\!\!\!\!\!&&\nabla\sigma_{i}(X_r, \alpha_r)J_{s,r}\, {\mathbf 1}_{\{ s\leq r\leq t\}}\nonumber\\
&&+\int^t_{r\vee s}\Big[\nabla^2b(X_{u}, \alpha_u)(D^i_r X_u, J_{s,u})+\nabla b(X_u, \alpha_u)D^{i}_{r}J_{s,u}\Big] du\nonumber\\
&&+\int^t_{r\vee s}\sum^{d}_{k=1}\Big[\nabla^2\sigma_{k}(X_u, \alpha_u)(D^i_r X_u, J_{s,u})
+\nabla\sigma_{k}(X_u, \alpha_u)D^{i}_{r}J_{s,u}\Big]dW^{k}_u,\nonumber\\
\end{eqnarray}
and for $r>t$, $D^{i}_{r}J_{s,t}=0$.
For $s\leq r\leq t$,    from the above identity we deduce the following estimates
\begin{eqnarray*}
|D^{i}_{r}J_{s,t} |^{p}\leq\!\!\!\!\!\!\!\!&&C|\nabla\sigma_{i}(X_r, \alpha_r)J_{s,r} |^p\nonumber\\
&&+C\left|\int^t_{r}\Big[\nabla^2b(X_{u}, \alpha_u)(D^i_r X_u, J_{s,u})+\nabla b(X_u, \alpha_u)D^{i}_{r}J_{s,u} \Big]du\right|^p\nonumber\\
&&+C\left|\int^t_{r}\Big[\nabla^2\sigma(X_u, \alpha_u)(D^i_r X_u, J_{s,u} )+\nabla\sigma(X_u, \alpha_u)D^{i}_{r}J_{s,u}\Big]dW_u\right|^p.
\end{eqnarray*}
Applying   H\"{o}lder's  and
Burkholder-Davis-Gundy's inequality, we have
\begin{eqnarray*}
\mathbb{E}\left[\sup_{r\leq t\leq T}|D^i_r J_{s,t} |^{p}\right]
\leq\!\!\!\!\!\!\!\!&&C|J_{s,r} |^{p}+C\int^T_r\mathbb{E}(|D^i_r X_u|^{p}\cdot|J_{s,u} |^{p})du
\\
&&
+C \int^T_r\mathbb{E}|D^i_r J_{s,u} |^{p}du\\
\leq\!\!\!\!\!\!\!\!&&C|J_{s,r} |^{p}+C\int^T_r(\mathbb{E}|D^i_r X_u|^{2p})^{1/2}\cdot(\mathbb{E}|J_{s,u} |^{2p})^{1/2}du\\
&&+C\int^T_r\mathbb{E}|D^i_r J_{s,u} |^{p}du.
\end{eqnarray*}
Hence, Lemma \ref{l.4.3} and Gronwall's inequality yield
\begin{equation}
\mathbb{E}\left[\sup_{r\leq t\leq T}|D^i_r J_{s,t} |^{p}\right]\leq Ce^{C(T-r)}.
\end{equation}
The case    $r< s$ can be handled in a similar way and hence the statement (i) is proved.

(ii) Note that for any  $j=1,\ldots, d$,
\[
D^{j}_{s}X_t =J_{s,t}\sigma_{j}(X_s, \alpha_s) \ \hbox{when $s\leq t$ and}\quad
 D^{j}_{s}X_t=0\  \hbox{when $s>t$} \,.
\]
An application of chain rule, which can be easily established,   yields
$$D^i_r(D^{j}_{s}X_t)=(D^i_r J_{s,t})\sigma_j(X_{s}, \alpha_s)+J_{s,t}\nabla\sigma_{j}(X_{s}, \alpha_s)D^j_r X_s.$$
Then, H\"older's inequality gives
\begin{eqnarray*}
 \mathbb{E}|D^i_r(D^{j}_{s}X_t)|^{p}\leq\!\!\!\!\!\!\!\!&& C\mathbb{E}\left|(D^i_r J_{s,t})\sigma_j(X_{s}, \alpha_s)\right|^{p}
+C\mathbb{E}\left|J_{s,t}\nabla\sigma_{j}(X_{s}, \alpha_s)D^j_r X_s\right|^{p}\\
\leq\!\!\!\!\!\!\!\!&&C \left(\mathbb{E}|D^i_r J_{s,t}|^{2p}\right)^{1/2}[\mathbb{E}(1+|X_{s}|^{2p})]^{1/2}\\
&&+C(\mathbb{E}|J_{s,t}|^{2p})^{1/2}(\mathbb{E}|D^i_r X_s|^{2p})^{1/2}
\leq  C,
\end{eqnarray*}
which   implies (ii).
\end{proof}

\begin{Rem}\label{r.4.5}Following the same  procedure as above we can prove that if Hypothesis \textbf{(${\bf H}_\infty$)} holds,   then
$J_{s, t}\in\mathbb{D}^{\infty}(\mathbb{R}^n\otimes\mathbb{R}^n\otimes V)$
and $ X_t\in\mathbb{D}^{\infty}(\mathbb{R}^n\otimes V)$.
\end{Rem}

\vspace{0.3cm}

We denote by     $(DX_t)^*$  the transpose of the random matrix  $DX_t$. From the relation between
 $DX_t$ and $J_{s, t}$,  we have
$(DX_t)^*(r) =\sigma(X_r, \alpha_r)^{\ast}J^{\ast}_{r,t} $.
Then, the Malliavin matrix $M_t$ of the random vector $X_t$ is defined by:
\begin{eqnarray*}
M_t =\!\!\!\!\!\!\!\!&&\langle DX_t \,, (DX_t)^*\rangle _H=\int^t_0 J_{s,t}\sigma(X_s, \alpha_s)\sigma(X_s, \alpha_s)^{\ast}J_{s,t}^{\ast}ds\\
=\!\!\!\!\!\!\!\!&&J_{0,t}\int^t_0 J^{-1}_{0,s}\sigma(X_s, \alpha_s)\sigma(X_s, \alpha_s)^{\ast}(J^{-1}_{0,s})^{\ast}ds J_{0,t}^{\ast} \\
=\!\!\!\!\!\!\!\!&&J_{0,t}C_t J_{0,t}^{\ast},
\end{eqnarray*}
where
\[
C_t=\int^t_0 J^{-1}_{0,s}\sigma(X_s, \alpha_s)\sigma(X_s, \alpha_s)^{\ast}(J^{-1}_{0,s})^{\ast}ds,
\]
is the so-called reduced Malliavin matrix of $X_t$.

\medskip
Our aim is to show that, under a suitable nondegeneracy condition on the coefficients,
 the Malliavin matrix $M_t$ is invertible $\mathbb{P}$-a.s. and the determinant of its inverse has
 negative moments of all orders.
  The difficulty in our current situation
is that the vector fields $b$ and $\si_1,\dots\,, \si_d$ depend on the Markovian switching process $\al_t$.
To overcome this difficulty we follow the following procedure inspired by  \cite{FLT}.

For $t \ge 0$ we  define $N_t:=N([0, t], m_0(m_0-1)K)$, so $\{N_t, t\ge 0\} $ is a Poisson process with
parameter $m_0(m_0-1)K$. Conditioned on the number of jumps of the Poisson process up to time $t$, that is, $N_t=k$,
 there exists a  random interval $[T_1, T_2]$ with $0\leq T_1< T_2\leq t$, such that
  $T_2-T_1\geq\frac{t}{k+1}$. This implies that $\alpha_t=\alpha_{T_1}$ for all $ t\in[T_1, T_2)$
  (because that the jump times of $\al_t$ are a subset of the jump times of $N_t$).
  On this random time interval, we will apply the classical techniques of Malliavin calculus.

To this end we need the following version of Norris lemma on time intervals.

\begin{lemma}\label{l.4.6}
Let  $t_1 \ge 0$ and let $\xi_1, \xi_2$ be two
$\cF_{t_1}$-measurable random variables. Suppose that $\beta(t), \gamma(t)=(\gamma_1(t),\dots, \gamma_d(t))$ and
$u(t)=(u_1(t),\dots, u_d(t))$ are $\cF_t$-adapted processes. For any $ t\geq t_1$, set
$$a(t)=\xi_1+\int^t_{t_1}\beta(s)ds+\sum^{d}_{i=1}\int^t_{t_1} \gamma_i(s)dW^{i}_s$$
$$Y(t)=\xi_2+\int^t_{t_1} a(s)ds+\sum^{d}_{i=1}\int^t_{t_1} u_i(s)dW^{i}_s$$ and assume that
for some $p\ge 2$ and $T>0$
\begin{equation}  \label{moment}
 \mathbb{E}\left(\sup_{t_1\leq t\leq  T}(|\beta(t)|+|\gamma(t)|+|a(t)|+|u(t)|)^{p}\right)<\infty.
\end{equation}
Consider $t_2\in[0,T]$ satisfying $t_2-t_1\geq c$ for some constant  $c>0$.
Then, for any $q>8$ and  $r>0$ such that $18r <q-8$, there exists $\varepsilon_0=  \delta_0 c^{\gamma_0}$, where the positive constants $\delta_0$ and $\gamma_0$ depend on $p$, $q$, $r$ and $T$, such that for all $\varepsilon\in(0,\varepsilon _0)$
$$\mathbb{P}\left\{\int^{t_2}_{t_1} Y^2_tdt<\varepsilon ^q, \int^{t_2}_{t_1}(|a(t)|^2+|u(t)|^2)dt
\geq \varepsilon \right\}\leq \varepsilon ^{rp}.$$
\end{lemma}
The proof is similar to the proof of \cite[Lemma 2.3.2]{N} and  we  omit the details. Just remark that the lower bound $c$ on the length of the interval $[t_1, t_2]$ is crucial in the proof.

 \medskip
We are going to impose a   uniform H\"{o}rmander's condition on the coefficients. To formulate this condition we need some notation.
Consider   the following sets of vector fields:
\begin{eqnarray*}
\Sigma_0=\!\!\!\!\!\!\!\!&&\{\sigma_1,\dots, \sigma_d\}\,,\\
\Sigma_n=\!\!\!\!\!\!\!\!&&\{[\sigma_k, V], k=0, \dots d, V\in\Sigma_{n-1}\},\quad  n\geq 1\,,\\
\Sigma=\!\!\!\!\!\!\!\!&&\cup^{\infty}_{n=0}\Sigma_n\,,
\end{eqnarray*}
where $\displaystyle  \sigma_0 =b-\frac{1}{2}\sum^{d}_{i=1}(\nabla \sigma_i)\sigma_i$ and  $[V,G]=(\nabla G)V -(\nabla V)G$ denotes the Lie bracket between two vector fields $V$ and $G$.

The following uniform H\"ormander's condition requires that the vector space spanned by $\{V(x,\alpha), V\in \Sigma\}$  is  $\RR^n$ for all $(x,\al)\in\RR^n\times \SS$ in a uniform way,
  where $V_j(x, \alpha)$ denotes the vector obtained by freezing the variables $x$ and $\alpha$ in the vector field $V_j$.

\medskip
\noindent \textbf{(UHC)} (Uniform H\"{o}rmander's condition)
Condition \textbf{(${\bf H}_\infty$)} holds and
 there exists an integer $j_0\geq 0$ and a constant $c>0$ such that
\begin{equation} \label{HC}
 \sum^{j_0}_{j=0}\sum_{V\in\Sigma_j}(v^{\ast}V(x, \alpha))^2\geq c,
\end{equation}
for all $x\in \RR^n$, $\alpha \in \SS$ and $ v\in \RR^n$ with $|v|=1$.

\begin{theorem}  \label{thm4.7} Assume that the uniform H\"ormander's condition  \textbf{(UHC)} holds.
Then for all  $ t>0$ the Malliavin matrix $M_t$ of the random vector $X_t$ is invertible $\mathbb{P}$-a.s.
and $\det(M^{-1}_t)\in L^p(\Omega)$ for all $p\geq 2$. As a consequence,  for any $t>0$, the law
of $X_t$ is absolutely continuous with respect to Lebesgues measure and the density
is  smooth.
\end{theorem}

\begin{proof}
We recall that $M_t=J_{0,t}C_t J_{0,t}^{\ast}$. By Lemma \ref{l.4.3}  it suffices to prove  that $\det(C^{-1}_t)\in L^p(\Omega)$ for all $p\ge 2$.

Recall that $\{N_t=N([0, t], m_0(m_0-1)K), t\ge 0\}$  is a Poisson process with parameter $\lambda:=m_0(m_0-1)K$. For a fixed $t>0$, conditioned on  $N_t=k$, there exists a  random interval $[T_1, T_2]$  such that $T_2-T_1\geq\frac{t}{k+1}$ and $\alpha_t=\alpha_{T_1}$ for all $ t\in[T_1, T_2)$.

We introduce the following sets of vector fields:
\begin{eqnarray*}
\Sigma'_0=\!\!\!\!\!\!\!\!&&\Sigma_0\,; \\
\Sigma'_n=\!\!\!\!\!\!\!\!&&\bigg\{[\sigma_k, V], k=1, \cdots d, V\in\Sigma'_{n-1};\\
&&\quad \quad [\sigma_0, V]+\frac{1}{2}\sum^d_{j=1}[\sigma_j, [\sigma_j, V]], V\in\Sigma'_{n-1}\bigg\},
\quad  n\geq 1\,; \\
\Sigma'=\!\!\!\!\!\!\!\!&&\bigcup^{\infty}_{n=0}\Sigma'_n\,.
\end{eqnarray*}
We denote by $\Sigma_n(x,\alpha)$ (resp. $\Sigma'_n(x,\alpha)$) the subset
of $\mathbb{R}^n$ obtained by freezing the variable $x, \alpha$ in the vector
fields of $\Sigma_n$ (resp. $\Sigma'_n$). Clearly, the vector spaces spanned
by $\Sigma(x, \alpha)$ or by $\Sigma'(x, \alpha)$ coincide.
By condition ({\bf UHC}), there exists an integer $j_0\geq 0$ and a $c>0$ such that
\begin{equation} \label{HC}
\inf_{x\in\RR^n}\inf_{\al\in\SS}\sum^{j_0}_{j=0}\sum_{V\in\Sigma'_j}(v^{\ast}V(x, \alpha))^2\geq c,
\end{equation}
for all  $|v|=1$.

For all $j=0, 1, \dots, j_0$, denote  $m(j)=2^{-4j}$ and define
$$E_j=\left\{\sum_{V\in\Sigma'_j}\int^{T_{2}}_{T_1}(v^{\ast}J^{-1}_{0, s}V(X_s, \alpha_s))^2 ds\leq \varepsilon ^{m(j)}\right\}.$$
Clearly $\{v^{\ast}C_{t} v\leq \varepsilon \}\subset E_0$. Consider the decomposition
$$E_0\subseteq (E_0 \cap E_1 ^{c})\cup(E_1\cap E_2^{c})\cup\cdots\cup(E_{j_0-1}\cap E_{j_0}^{c})\cup F,$$
where $F=E_0\cap E_1\cap\cdots\cap E_{j_0}$. Then for any unit vector $v$ we have
\begin{eqnarray*}
\mathbb{P}\{v^{\ast}C_t v\leq \varepsilon| N_t=k \}\leq\!\!\!\!\!\!\!\!&&\mathbb{P}(E_0| N_t=k)\nonumber\\
\leq\!\!\!\!\!\!\!\!&& \mathbb{P}(F| N_t=k)+\sum^{j_0-1}_{j=0}\mathbb{P}(E_j\cap E_{j+1}^{c}| N_t=k)\,.
\label{e.4.7}
\end{eqnarray*}
We are going to estimate each term in the above    sum. This will be done in two steps.\\

\noindent
\textit{Step 1}: We can write
\begin{equation}
\mathbb{P}(F| N_t=k)\leq \mathbb{P}(F\cap G| N_t=k)+\mathbb{P}(G^c| N_t=k),
\label{e.4.8}
\end{equation}
where $G:=\{\sup_{T_1\leq s\leq T_2}\|J_{0,s}\|\leq \frac{1}{\varepsilon^{\beta}}\}$, $0<2\beta<m(j_0)$.  First we claim that when  $\varepsilon $ is sufficiently small,
the intersection $F\cap G \cap \{N_t=k\}$ is empty.
In fact,  taking into account the estimate (\ref{HC}),  on $N_t=k$, we have
\begin{eqnarray}
&&\sum^{j_0}_{j=0}\sum_{V\in\Sigma'_j}\int^{T_2}_{T_1}(v^{\ast}J^{-1}_{0, s}V(X_s, \alpha_s))^2 ds\nonumber\\
=\!\!\!\!\!\!\!\!&&\sum^{j_0}_{j=0}\sum_{V\in\Sigma'_j}\int^{T_2}_{T_1}
\left(\frac{v^{\ast}J^{-1}_{0, s}V(X_s, \alpha_s)}{|v^{\ast}J^{-1}_{0, s}|}\right)^2
|v^{\ast}J^{-1}_{0, s}|^2 ds\geq\frac{tc\varepsilon ^{2\beta}}{k+1},\label{e.4.9}
\end{eqnarray}
because   $|v^{\ast}J^{-1}_{0, s}|\geq \frac{1}{\|J_{0,s}\|}\geq \varepsilon^{\beta}$, and  $T_2-T_1 \ge \frac t{k+1}$.
On the other hand, the left-hand side of  \eref{e.4.9}
is bounded by $(j_0+1)\varepsilon ^{m(j_0)}$ on the set $F$.  Thus     $F\cap G \cap \{N_t=k\}=\emptyset$, provided $\varepsilon < \varepsilon_1$, where $\varepsilon_1=   [\frac {tc}{(k+1)(j_0+1)} ]^{ \frac 1{ m(j_0) -2\beta}}$.

Now we consider the second term in \eref{e.4.8}. Using Chebyshev  inequality we obtain
\[
\mathbb{P} \left( \sup_{T_1\le s\le T_2} | J_{0,s}| \ge \varepsilon ^{-\beta}\Big| N_t=k \right)
\le \varepsilon^{p\beta} \mathbb{E}  \left (\sup_{T_1\le s\le T_2} | J_{0,s}|  ^p \Big| N_t=k \right).
\]
Taking into account that the Poisson random measure $N$ is independent of the Brownian motion $W$,
we can estimate the above  conditional expectation using Burkholder-Davis-Gundy's inequality as in
Lemma  \ref{l.4.3}, and we obtain the estimate
\begin{equation} \label{j2}
\mathbb{P} \left( \sup_{T_1\le s\le T_2} | J_{0,s}| \ge \varepsilon ^{-\beta} \Big| N_t=k \right)
\le C  \varepsilon^{p\beta}.
\end{equation}

\medskip
\noindent
\textit{Step 2}: We shall bound the remaining parts. For any $j=0,\dots, j_0-1$ we have
\begin{eqnarray*}
&&\mathbb{P}(E_j\cap E_{j+1}^{c}| N_t=k)\\
=\!\!\!\!\!\!\!\!&&\mathbb{P}\left\{\sum_{V\in\Sigma'_{j}}\int^{T_2}_{T_1}(v^{\ast}J^{-1}_{0,s}V(X_s, \alpha_s))^2ds\leq \varepsilon ^{m(j)}\right.,\\
&&\quad \quad \left.\sum_{V\in\Sigma'_{j+1}}\int^{T_2}_{T_1}(v^{\ast}J^{-1}_{0,s}V(X_s, \alpha_s))^2ds>\varepsilon ^{m(j+1)}\Big| N_t=k\right\}\\
\leq\!\!\!\!\!\!\!\!&&\sum_{V\in\Sigma'_{j}}
\mathbb{P}\left\{\int^{T_{2}}_{T_{1}}(v^{\ast}J^{-1}_{0,s}V(X_s, \alpha_{T_1}))^2ds\leq \varepsilon ^{m(j)}\right.,\\
&&\quad \quad \quad \quad \sum^{d}_{i=1}\int^{T_2}_{T_1}(v^{\ast}J^{-1}_{0,s}[\sigma_i, V](X_s, \alpha_{T_1}))^2ds
+\int^{T_{2}}_{T_{1}}\left(v^{\ast}\right.J^{-1}_{0,s}\left([\sigma_0, V]\right.\\
&&\quad \quad \quad \quad +\frac{1}{2}\sum^{d}_{i=1}\left.[\sigma_i, [\sigma_i, V]])(X_s, \alpha_{T_1})\right)^2ds>
\left.\frac{\varepsilon ^{m(j+1)}}{n(j)}\Big|N_t=k\right\}\,,
\end{eqnarray*}
where $n(j)$ denotes the cardinality of the set $\Sigma'_j$.  Consider the continuous semimartingale
$\{v^{\ast}J^{-1}_{0,t}V(X_t, \alpha_{T_1}), T_1\leq t< T_2 \}$.
For any $t\in [T_1, T_2)$ It\^{o}'s formula yields
\begin{eqnarray*}
&&v^{\ast}J^{-1}_{0,t}V(X_t, \alpha_{T_1})\\
=\!\!\!\!\!\!\!\!&&v^{\ast}J^{-1}_{0,T_1}V(X_{T_1}, \alpha_{T_1})+\int^{t}_{T_1} v^{\ast}J^{-1}_{0,s}
\sum^{d}_{i=1}[\sigma_i, V](X_s, \alpha_{T_1})dW^i_s\\
\!\!\!\!\!\!\!\!&&+\int^t_{T_1} v^{\ast}J^{-1}_{0,s}\left\{[\sigma_0, V]+\frac{1}{2}\sum^d_{i=1}[\sigma_i, [\sigma_i, V]]\right\}(X_s, \alpha_{T_1})ds.
\end{eqnarray*}
Notice that $8m(j+1)<m(j)$ and also notice condition (\ref{moment}) holds for any $p$ and the fact that the Poisson random measure $N$ is independent of $W$.
An application of Lemma \ref{l.4.6} to the semimartingale $Y_t=v^{\ast}J^{-1}_{0,t}V(X_t, \alpha_{T_k})$ with time interval $[T_1, T_2]$ which satisfy $T_2-T_1\geq\frac{t}{k+1}$ on the set $N_t=k$ yields
\begin{equation}  \label{j1}
\mathbb{P}  (E_j\cap E_{j+1}^c | N_t=k) \le \varepsilon^p
\end{equation}
for any $p\ge 2 $, and for $ \varepsilon<   \varepsilon_0 $, where $\varepsilon_0= \delta_0(\frac{t}{k+1})^{\gamma_0} $.
The exponents $\delta_0$ and $\gamma_0$ only depend on $p$ and $T$. Therefore, from (\ref{j2}) and (\ref{j1}) we obtain
 \[
 \mathbb{P}\{ v^{\ast}C_t v\leq \varepsilon| N_t=k\}\leq   \varepsilon^{p},
 \]
for any $p\ge 2 $, and for $ \varepsilon<\min(\varepsilon_0, \varepsilon_1) $.
Then, following the steps of \cite[Lemma 2.3.1]{N}, we can obtain that
$$
\mathbb{P}\left\{\inf_{|v|=1}v^{\ast}C_t v\leq \varepsilon\Big| N_t=k\right\}\leq   \varepsilon^{p},
$$
for all $0<\varepsilon\leq C_1 (\frac{t}{k+1})^{C_2} $ and for all $p\geq 2$, where $C_1$ and $C_2$ are positive constants depending on $p$, $T$ and $n$. Consequently,
\begin{eqnarray*}
\EE  |\det(C_t)|^{-p} \le\!\!\!\!\!\!\!\!&& \EE  (\inf_{|v|=1}v^{\ast}C_t v)^{-np} \\
\le\!\!\!\!\!\!\!\!&& \sum^{\infty}_{k=0}\mathbb{P}(N_t=k)\EE\left[ \left(\inf_{|v|=1}v^{\ast}C_t v\right)^{-np}\Big|N_t=k\right]\\
\le\!\!\!\!\!\!\!\!&&\sum^{\infty}_{k=0}\frac{\lambda^k}{k!}e^{\lambda}\left[  C_1 \left( \frac t{k+1} \right)^{C_2}
+  \frac{1}{C_1} \left( \frac {k+1}{t} \right)^{C_2} \right] <\infty.
\end{eqnarray*}
The proof is now complete.
\end{proof}

\section{Bismut  type  formula}

In this section, we  prove a version of Bismut  type  formula for SDEs with Markovian switching. As an application, this  formula is  used
to obtain the strong Feller property for the transition semigroup of $(X_t, \alpha_t)$.

\begin{theorem}
Suppose the condition \textbf{(UHC)} holds.
Then for any $f\in C^{2}_b(\mathbb{R}^n\times\mathbb{S})$, we have
\begin{equation} \label{Bi}
 \nabla  P_tf(x, \alpha)
=\mathbb{E}\left[f(X_t, \alpha_t)\int^t_0\sigma(X_s, \alpha_s)^{\ast}J_{s,t}^{\ast}M^{-1}_t J_{0,t}  d W_s\right],
\end{equation}
where $M_t=\int^t_0 J_{s,t}\sigma(X_s, \alpha_s)\sigma(X_s, \alpha_s)^{\ast}J_{s,t}^{\ast}ds$ and the stochastic integral is interpreted in the Skorohod sense, that is,  $\int^t_0\sigma(X_s, \alpha_s)^{\ast}J_{s,t}^{\ast}M^{-1}_t J_{0,t}  d W_s $ is the divergence of the process  $ \{\sigma(X_s, \alpha_s)^{\ast}J_{s,t}^{\ast}M^{-1}_t J_{0,t}   I_{[0,t]}(s), s\ge 0\}$.
\end{theorem}

\begin{proof}
For any $\xi\in\mathbb{R}^{n}$ with  $|\xi|=1$, let $h^{\xi}=(DX_t)^{\ast} M^{-1}_t J_{0,t}\xi$. Then we get
\begin{eqnarray*}
\langle DX_t, h^{\xi}\rangle_{H}
=\!\!\!\!\!\!\!\!&& \langle DX_t, (DX_t)^{\ast}M^{-1}  J_{0,t}\xi\rangle_{H}\\
=\!\!\!\!\!\!\!\!&&\langle DX_t, (DX_t)^{\ast} \rangle_{H} M^{-1}  J_{0,t}\xi
=J_{0,t}\xi.\nonumber
\end{eqnarray*}
We claim that $h^{\xi}\in \mathbb{D}^{1, p}(H\otimes V)$ for any $p\geq 2$.
In fact, we have
\begin{eqnarray*}
D^{i}_{s}h^{\xi}=\!\!\!\!\!\!\!\!&&(D^{i}_{s}(DX_t)^{\ast})M^{-1}_t
J_{0,t}\xi+(DX_t)^{\ast}M^{-1}_t(D^{i}_{s}J_{0,t})\xi+(DX_t)^{\ast}(D^{i}_{s}M^{-1}_t) J_{0,t}\xi\nonumber\\
=\!\!\!\!\!\!\!\!&&(D^{i}_{s}(DX_t)^{\ast})M^{-1}_t J_{0,t}\xi
+(DX_t)^{\ast}M^{-1}_t(D^{i}_{s}J_{0,t})\xi\nonumber\\
&&-(DX_t)^{\ast}M^{-1}_t\left[ \langle D^{i}_{s}(DX_t) ,  (DX_t)^{\ast} \rangle_H + \langle DX_t, D^{i}_{s}(DX_t)^{\ast} \rangle_H \right]M^{-1}_t J_{0,t}\xi.
\end{eqnarray*}
 By Lemma \ref{l.4.3}, Lemma \ref{l.4.4} and Theorem  \ref{thm4.7}    we obtain
\[
\mathbb{E}_{1}\|h^{\xi}\|^{p}_{H\otimes V}+\mathbb{E}_{1}\|Dh^{\xi}\|^{p}_{H\otimes H\otimes V}
\leq\mathbb{E}\|h^{\xi}\|^{p}_{H}+\sum^{d}_{i=1}\mathbb{E}\int^{t}_{0}\|D^{i}_{s}h^{\xi}\|^{p}_{H}ds<\infty.
\]
Notice that
 $h^{\xi}_s=\sigma(X_s, \alpha_s)^{\ast}J_{s,t}^{\ast}M^{-1}_t J_{0,t}\xi$.
Then, the derivative of  $P_tf(x, \alpha)$ can be computed as follows
\begin{eqnarray}
\langle \nabla P_tf(x,\alpha), \xi\rangle=\!\!\!\!\!\!\!\!&&\mathbb{E}(\nabla^{\xi}[f(X_t, \alpha_t)])
=\mathbb{E}\left[\nabla f(X_t, \alpha_t)J_{0,t}\xi\right]\nonumber\\
=\!\!\!\!\!\!\!\!&&\mathbb{E}\left[\nabla f(X_t, \alpha_t)\langle DX_t, h^{\xi}\rangle_{H}\right]
=\mathbb{E}\left[\langle Df(X_t, \alpha_t), h^{\xi}\rangle_H\right]\nonumber\\
=\!\!\!\!\!\!\!\!&&\mathbb{E}_{1}\left[\langle Df(X_t, \alpha_t), h^{\xi}\rangle_{H\otimes V}\right]
=\mathbb{E}_{1}\left[\langle f(X_t, \alpha_t),\delta(h^{\xi})\rangle_{V}\right]\nonumber\\
=\!\!\!\!\!\!\!\!&&\mathbb{E}\left[f(X_t, \alpha_t)\delta(h^{\xi})\right]\nonumber\\
=\!\!\!\!\!\!\!\!&&\mathbb{E}\left[f(X_t, \alpha_t)\int^t_0\sigma(X_s, \alpha_s)^{\ast}J_{s,t}^{\ast}M^{-1}_t J_{0,t}\xi d W_s\right],\nonumber
\end{eqnarray}
where the second and forth equalities  follow from the chain rule and
the sixth equality follows from the integration by parts formula, where
the stochastic integral  is interpreted  in the Skorohod  sense.
\end{proof}

\vspace{0.3cm}

As an application of the above  Bismut  type  formula, we intend to prove the strong Feller property. That is, we claim that  for any $t>0$ and for
any bounded Borel measurable function $f$ on ${\mathbb{R}^n}\times\mathbb{S}$, $P_tf(x, \alpha)$
is bounded and continuous in   $(x, \alpha)$.
Since  $\mathbb{S}$ is a finite set,
 it is sufficient  to prove that  for any $\alpha\in\mathbb{S}$\,,
$P_tf(x, \alpha)$ is bounded and continuous with respect to $x$.
%We shall prove that this is a straightforward consequence of the Bismut formula.
\begin{theorem} Suppose that  condition \textbf{(UHC)} holds.
Then for any $f\in{\mathcal B}_{b}({\mathbb{R}^n}\times\mathbb{S})$,
$t>0$, $\alpha\in\mathbb{S}$, $x\in \mathbb{R}^n$, we have
\begin{equation}
\lim_{y\rightarrow x}|P_tf(y, \alpha)-P_tf(x, \alpha)|=0.\nonumber
\end{equation}
\end{theorem}
\begin{proof}
Fix $x\in\mathbb{R}^n$, $ \alpha\in\mathbb{S}$, $ t>0$.  First, we
 consider the case  where   $f\in C^{2}_b(\mathbb{R}^n\times\mathbb{S})$.
Applying  (\ref{Bi}), we have
\begin{eqnarray*}
|\nabla P_tf(x, \alpha)|=\!\!\!\!\!\!\!\!&&\sup_{|\xi|=1}|\langle \nabla P_tf(x,\alpha), \xi\rangle|\nonumber\\
\leq\!\!\!\!\!\!\!\!&&\sup_{|\xi|=1}\mathbb{E}\left|f(X_t, \alpha_t)\int^t_0h^{\xi}_s d W_s\right|\nonumber\\
\leq\!\!\!\!\!\!\!\!&&\|f\|_{\infty}\sup_{|\xi|=1}\mathbb{E}_1\left\|\int^t_0h^{\xi}_s d W_s\right\|_{V}.\nonumber\\
\end{eqnarray*}
where $h^{\xi}_s=\sigma(X_s, \alpha_s)^{\ast}J_{s,t}^{\ast}M^{-1}_t J_{0,t}\xi$. As we have seen, since the process $h^{\xi}_s$ is not adapted, the integral    is  Skorohod integral   and it can be estimated as follows:
\begin{eqnarray*}
\mathbb{E}_1\left\|\int^t_0 h^{\xi}_s d W_s\right\|_{V}
\leq\!\!\!\!\!\!\!\!&&  \left(\mathbb{E}_1\left\|\int^t_0 h^{\xi}_s d W_s\right\|^{2}_{V}\right)^{1/2}\\
=\!\!\!\!\!\!\!\!&& \mathbb{E}_1\int^t_0 \|h^{\xi}_s\|^{2}_{\mathbb{R}^d \otimes V}d s+
\mathbb{E}_1\int^{t}_{0}\int^{t}_{0}\langle D_{r}h^{\xi}_s,D_{s}h^{\xi}_r\rangle _{  \mathbb{R}^d \otimes \mathbb{R}^d \otimes V } drds\\
\leq\!\!\!\!\!\!\!\!&&\mathbb{E}_1\int^t_0 \|h^{\xi}_s\|^{2}_{\RR^{d}\times V}d s+
\mathbb{E}_1\int^{t}_{0}\int^{t}_{0}\|D_{r}h^{\xi}_s\|^{2}_{\mathbb{R}^{d}\otimes\mathbb{R}^{d}\otimes V}drds\\
=\!\!\!\!\!\!\!\!&&\mathbb{E}\|h^{\xi}\|^{2}_{H}+\sum^{d}_{i=1}\mathbb{E}\int^{t}_{0}\|D^{i}_{s}h^{\xi}\|^{2}_{H}ds.
\end{eqnarray*}
Thus, we have
$$
\|\nabla P_tf(x, \alpha)\|\leq C_{x}\|f\|_{\infty},
$$
where the constant $C_x$ depends on the initial condition $x$.
In fact, we also have that for any   $y\in B_r(x) =\{y\in\mathbb{R}^n: |y-x|\leq r\}$,  the following
inequality holds
\begin{eqnarray*}
\sup_{y\in B_r(x)}|\nabla P_tf(y, \alpha)|\leq C_{x, r}\|f\|_{\infty}.
\end{eqnarray*}
This implies easily for any $|y-x|\leq 1$,
\begin{eqnarray}
|P_t f(y, \alpha)-P_t f(x, \alpha)|\le\!\!\!\!\!\!\!\!&&C_{x, 1}\|f\|_{\infty}|y-x|.\nonumber
\end{eqnarray}
Hence, the theorem holds for any $f\in{\mathcal B}_b({\mathbb{R}^n}\times\mathbb{S})$ by a standard argument.
\end{proof}

\def\refname{References}


\begin{thebibliography}{2}

%\bibitem {BWY1} J. Bao, F. Wang, C. Yuan,  \emph{Derivative formula and Harnack inequality for degenerate functional SDEs},
%Stoch. Dyn. 13 (2013) 1250013.

%\bibitem {BWY2} J. Bao, F. Wang, C. Yuan,  \emph{Bismut formulae and applications for functional SPDEs},
%Bull. Sci. math. 137 (2013) 509-522.

\bibitem {BBG} G. Basak, A. Bisi, M. Ghosh,  \emph{Stability of a Random Diffusion with Linear Drift},
J. Math. Anal. Appl. 202, 604-622, 1996.

%\bibitem{C} T. Cass, \emph{Smooth densities for solutions to stochastic differential equations with jump}, Stochastic Process. Appl.
%119 (2009) 1416-1435.

%\bibitem{Chen} M. Chen, \emph{From Markov Chains to Non-equilibrium Particle Systems},
%World Scientific, Singapore, 2004.

%\bibitem{DZ} G. Da Prato and J. Zabczyk, \emph{Ergodicity for infinite dimensional systems},
%Cambridge University Press, 1996.

\bibitem{FLT} B. Forster, E. L\"{u}tkebohmert and J. Teichmann, \emph{Absolutely continuous laws of jump-diffusions in finite and infinite dimensions with appliationc to mathematical finance}, SIAM J. Math. Anal. 40(5), 2132-2153, 2009.

%\bibitem{DX} Z. Dong and Y. Xie, \emph{Ergodicity of linear SPDE driven by L\'evy Noise}, J. Syst. Sci. Complex
%23 (2010) 137-152.

%\bibitem {GW} A. Guillin, F.-Y. Wang, \emph{Degenerate Fokker - Planck equations: Bismut formula, gradient estimate and Harnack inequality},
%J. Differential Equations 253 (2012) 20-40.

%\bibitem {LDS} C. Li, Z. Dong, R. Situ. \emph{Almost sure stability of linear stochastic differential equations with jumps},
%Probab. theory Related Fields 123 (2002) 121-155.

\bibitem {Ma} P. Malliavin, \emph{Stochastic analysis},
Springer-Verlag, Berlin, 1997.

\bibitem {M} X. Mao, \emph{Stability of stochastic differential equations with Markovian switching},
Stochastic Process. Appl. 79: 45-67, 1999.

%\bibitem{NJ} J. Norris, \emph{Simplified Malliavin calculus},
%Seminaire de probalilities (Strasbourg), 20 (1986), 101-130.

\bibitem{N} D. Nualart, \emph{The Malliavin Calculus and Related Topics},
Springer, Berlin, 2006.

%\bibitem{oksendal}  Di Nunno, G.; \O ksendal, B. and  Proske, F.
% Malliavin calculus for L\'evy processes with applications to finance. Universitext. Springer, Berlin, 2009.

%\bibitem{P} E. Priola, \emph{Formulae for the derivatives of degenerate diffusion semigroups}, J. Evol. Equ. 6 (2006), 557-600.

%\bibitem{Rey} L. Rey-Bellet, \emph{Ergodic properties of Markov processes},
%Open quantum systems. II, 1-39, Lecture Notes in Math., 1881, Springer, Berlin, 2006.

%\bibitem{WZ} F. Wang, X. Zhang, \emph{Derivative formula and applications for degenerate diffusion semigroup}, J. Math. Pures Appl.
%99 (2013) 726-740.

%\bibitem{Xi} F. Xi, \emph{Feller property and exponential ergodicity of diffusion process with state-dependent switching},
%Science in China Series A: Mathematics, 51(3) (2008), 329-342.

%\bibitem{XY1} F. Xi and G. Yin, \emph{Jump-diffusions with state-dependent switching: existence and uniqueness, Feller property,
%linearization, and uniform ergodicity}, Science in China Series A: Mathematics, 54(12) (2011), 2651¨C2667.

%\bibitem{XY} F. Xi and G. Yin, \emph{The strong Feller property of switching jump-diffusion processes},
%Stat Probab Letters, 83 (2013) 761-767.

%\bibitem{XZ} F. Xi and L. Zhao, \emph{On the stability of diffusion processes with state-dependent switching},
%Science in China Series A: Mathematics, 49(9) (2006), 1258-1274.

\bibitem{YZ} G. Yin and C. Zhu, \emph{Hybrid switching diffusions: Properties and applications},
Springer, New York, 2010.

\bibitem {YM} C. Yuan and X. Mao, \emph{Asymptotic stability in distribution of stochastic differential equations with Markovian switching},
Stochastic Process. Appl. 103: 277-291, 2003.

%\bibitem{ZY} C. Zhu and G. Yin, \emph{On the strong Feller, recurrence, and weak stabilization of regime-switching diffusion},
%SIAM J. Control Optim. 48 (2009) 2003-2031.

%\bibitem{Z} X. Zhang, \emph{Stochastic flows and Bismut formulas for stochastic Hamiltonian systems}, Stochastic Process. Appl.
%120 (2010) 1929-1949.

\end{thebibliography}
\end{document}